\documentclass[11pt,a4paper]{amsart}
\usepackage[margin=3.3cm]{geometry}
\usepackage{amssymb}
\usepackage{graphicx} 
\usepackage[mathscr]{euscript}
\usepackage{enumerate}
\usepackage{xspace}
\usepackage{color}
\usepackage{enumerate}
\usepackage{enumitem}
\usepackage{amsthm}
\usepackage{dsfont}
\usepackage{bbm}
\usepackage[colorlinks=true, linkcolor = blue, citecolor = blue]{hyperref}
\usepackage[numbers]{natbib}
\usepackage[backgroundcolor=white,bordercolor=red]{todonotes}
\DeclareMathAlphabet{\mathpzc}{OT1}{pzc}{m}{it}
\usepackage[nameinlink]{cleveref}
\numberwithin{equation}{section}
\begin{document}

\theoremstyle{plain}
\newtheorem{theorem}{Theorem}[section]
\newtheorem{lemma}[theorem]{Lemma}
\newtheorem{corollary}[theorem]{Corollary}
\newtheorem{definition}[theorem]{Definition}
\theoremstyle{definition}
\newtheorem{remark}[theorem]{Remark}

\crefname{lemma}{lemma}{lemmas}
\Crefname{lemma}{Lemma}{Lemmata}
\crefname{corollary}{corollary}{corollaries}
\Crefname{corollary}{Corollary}{Corollaries}

\newcommand{\cadlag}{c\`adl\`ag }
\newcommand{\1}{\mathbb{I}}
\newcommand{\X}{\mathsf{X}} 
\newcommand{\D}{{\mathbb{D}}}
\newcommand{\Df}{\mathbf{D}}
\renewcommand{\epsilon}{\varepsilon}
\newcommand{\bcdot}{\boldsymbol{\cdot}}
\renewcommand{\u}{\sigma}
\newcommand{\w}{\mu} 
\newcommand{\B}{\mathbb{B}}
\newcommand{\h}{h^*}

\title[Dual Yamada--Watanabe Theorem for L\'evy driven SDEs]{A Dual Yamada--Watanabe Theorem for L\'evy driven stochastic differential equations} 
\author[D. Criens]{David Criens}
\address{D. Criens - University of Freiburg, Ernst-Zermelo-Str. 1, 79104 Freiburg, Germany}
\email{david.criens@stochastik.uni-freiburg.de}

\keywords{Yamada--Watanabe Theorem; L\'evy Process; Stochastic Differential Equation; Joint Uniqueness; Weak Uniqueness; Weak Existence; Strong Uniqueness; Strong Existence; Martingale Problem\vspace{1ex}}

\subjclass[2010]{60H10, 60G51, 60H05}

\thanks{The author thanks the anonymous referee for many helpful comments.}
\thanks{Financial support from the DFG project No. SCHM 2160/15-1 is gratefully acknowledged.}

\date{\today}
\maketitle

\frenchspacing
\pagestyle{myheadings}

\begin{abstract}
We prove a dual Yamada--Watanabe theorem for one-dimensional stochastic differential equations driven by quasi-left continuous semimartingales with independent increments. In particular, our result covers stochastic differential equations driven by (time-inhomogeneous) L\'evy processes. More precisely, we prove that weak uniqueness, i.e. uniqueness in law, implies weak  joint uniqueness, i.e. joint uniqueness in law for the solution process and its driver.
\end{abstract}

\section{Introduction}
The classical Yamada--Watanabe theorem \cite{YW} for Brownian stochastic differential equations (SDEs) tells us that strong uniqueness (i.e. pathwise uniqueness) and weak existence implies weak uniqueness (i.e. uniqueness in law) and strong existence. Jacod \cite{J80} lifted this result to SDEs driven by semimartingales and extended it by showing that strong uniqueness and weak existence is equivalent to weak joint uniqueness and strong existence. For Brownian SDEs, results related to Jacod's theorem have also been shown by Engelbert \cite{doi:10.1080/17442509108833718}.

For a moment consider the SDE
\[
d X_t = \u_t (X) d L_t, \quad L = \text{L\'evy process},
\]
with non-degenerate coefficient \(\u\), i.e. \(\u \not = 0\). Since in this case
\begin{align}\label{eq: intro}
d L_t = \frac{\u_t (X)d L_t}{\u_t (X)} = \frac{d X_t}{\u_t (X)},
\end{align}
the noise \(L\) can be recovered from the solution \(X\) and consequently, weak uniqueness implies weak joint uniqueness. By Jacod's theorem, this observation shows that weak uniqueness and strong existence implies strong uniqueness and weak existence, which can be seen as a \emph{dual Yamada--Watanabe theorem}. It is a natural and interesting question whether the dual theorem also holds for SDEs with possibly degenerate coefficients, i.e. in case \(\u= 0\) is allowed.
Cherny \cite{doi:10.1137/S0040585X97979093} answered this question 
affirmatively for Brownian SDEs by proving that weak uniqueness implies weak joint uniqueness.

More recently, the dual theorem has been generalized to various infinite-dimensional Brownian frameworks, see \cite{doi:10.1142/S0219493710002991, rem20, tappe20}.  
A version for SDEs with L\'evy drivers seems to be missing in the literature. To the best of our knowledge, the only formulation of a dual theorem for SDEs with discontinuous noise is the main result of \cite{DYWTPP20}, which is a version for SDEs driven by time-inhomogeneous Poisson processes. However, the proof in  \cite{DYWTPP20} has a gap.\footnote{The first sentence in the proof of \cite[Lemma 3.3]{DYWTPP20} is the following: "By Lemma 2.3, we only need to prove  that two Poisson processes are independent if and only if their sum is also a Poisson process." The if implication, which is supposed to prove a direction of \cite[Lemma 3.3]{DYWTPP20} used in the proof of the main result \cite[Theorem 3.1]{DYWTPP20}, is not true. Indeed, Jacod \cite{10.2307/3212423} gave an example of two dependent Poisson processes whose sum is also a Poisson process. The gap in the proof of \cite[Lemma 3.3]{DYWTPP20} is that the factorization \(E[e^{i u (X + Y)}] = E[e^{iu X}] E[e^{iu Y}]\) does not suffice to conclude independence of \(X\) and \(Y\). \label{fn: incomplet}}

The purpose of this short paper is to close the gap in the literature and to prove a dual theorem for SDEs driven by quasi-left continuous semimartingales with independent increments (SIIs), which is a large class of drivers including in particular all L\'evy processes.  We now directly formulate our main result, where we refer to the next section for precise definitions. Let \(\w\) and \(\u\) be real-valued predictable processes on the path space of \cadlag functions \(\mathbb{R}_+ \to \mathbb{R}\) and let \(L\) be an SII, which we parameterize below by its deterministic semimartingale characteristics.
\begin{theorem}\label{theo: main1intro}
	For the SDE 
	\begin{align}\label{eq: mainSDE}
	d X_t = \w_t (X) dt + \u_t (X) d L_t, \qquad X_0 = x_0 \in \mathbb{R},
	\end{align}
	the following are equivalent:
	\begin{enumerate}
		\item[\textup{(i)}] Strong uniqueness and weak existence holds.
		\item[\textup{(ii)}] Weak uniqueness and strong existence holds.
		\item[\textup{(iii)}] Weak joint uniqueness and strong existence holds. 
	\end{enumerate}
\end{theorem} 

We shortly comment on the proof of \Cref{theo: main1intro}. The implication (i) \(\Rightarrow\) (ii) is the classical Yamada--Watanabe theorem and the equivalence of (i) and (iii) is Jacod's theorem. Thus, it suffices to prove (ii) \(\Rightarrow\) (iii). More precisely, we prove that weak uniqueness implies weak joint uniqueness, see \Cref{prop: main1} below. Similar to Cherny's proof for Brownian SDEs, the main idea is to recover the driver \(L\) from the solution process \(X\) and an independent SII \(V\). The technical core of the argument is the proof for independence of \(X\) and \(V\). We first show that the laws of \(X\) and \(V\) can be characterized via certain martingale problems and that the covariation between the corresponding test martingales vanishes. This observation allows us to deduce independence from the weak uniqueness assumption with a change of measure argument, which is inspired by ideas from \cite{EK}. 
Our proof for the independence of \(X\) and \(V\) is different from Cherny's argument, which uses a second auxiliary process. The proofs of the dual theorems in the recent works \cite{doi:10.1142/S0219493710002991, rem20,DYWTPP20} are based on Cherny's argument, while in different settings, and the proof in \cite{tappe20} uses results from \cite{rem20} and the so-called method of the moving frame.

Let us also mention an interesting follow up question: It is well-known that the Yamada--Watanabe theorem remains true for SDEs driven by a Poisson random measure, see \cite[Theorem 14.94]{J79}. Thus, it is only natural to ask whether the dual theorem also holds for such SDEs. We leave this question for future research.

Finally, we end the introduction with a short comment on our notation: As far as possible, we use the notation developed in the monograph \cite{JS} by Jacod and Shiryaev. In particular, we use the Strasbourg notation for (stochastic) integrals, i.e. we write \(H \bcdot X\) for the stochastic integral of \(H\) with respect to \(X\). Furthermore, we stress that throughout the paper equalities hold up to a null set.

The article is structured as follows: In \Cref{sec: 2} we provide definitions and the statement that weak uniqueness implies weak joint uniqueness. The proof of this claim is given in \Cref{sec: 3}.

\section{The Setting and Main Results}\label{sec: 2}
In this section we introduce our setting in a precise manner. 
Let us start with a description of the random driver for the SDE \eqref{eq: mainSDE}. 
A real-valued \cadlag adapted process with initial value zero is called \emph{SII}, if it is a quasi-left continuous semimartingale whose semimartingale characteristics \((B, C, \nu)\) (relative to a fixed  truncation function \(h \colon \mathbb{R} \to \mathbb{R}\)) are deterministic. It is well-known (\cite[Proposition II.2.9]{JS}) that the semimartingale characteristics \((B, C, \nu)\) of an SII admit a decomposition
\[
dB_t = b_t dA_t, \quad dC_t = c_t dA_t, \quad \nu(dx, dt) = F_t (dx) d A_t,
\]
for an increasing continuous function \(A \colon \mathbb{R}_+ \to \mathbb{R}_+\) starting in the origin, \(d A_t\)-integrable Borel functions \(b \colon \mathbb{R}_+ \to \mathbb{R}\) and \(c \colon \mathbb{R}_+ \to \mathbb{R}_+\), and a Borel transition kernel \(F\) from \(\mathbb{R}_+\) into \(\mathbb{R}\) with \(F_t (\{0\}) = 0\) and \(\int_0^\cdot \int(1 \wedge |y|^2) F_s (dy) d A_s < \infty\).
We call such a quadruple \((b, c, F; A)\) \emph{local characteristics} of the SII. Clearly, the local characteristics are not unique. Nevertheless, it is well-known (\cite[Theorem II.4.25]{JS}) that the law of an SII is fully characterized by its semimartingale characteristics and therefore also by its local characteristics. 
The class of SIIs can equivalently be defined as the class of quasi-left continuous semimartingales with independent increments, see \cite[Theorem II.4.15]{JS}. This characterization explains the abbreviation \emph{SII} which stands for \emph{semimartingale with independent increments}.
It is clear that a L\'evy process with (time-independent) L\'evy--Khinchine triplet \((b, c, F)\) is an SII  and (a version of) its local characteristics is given by \((b, c,  F; A_t = t)\).

Next, we introduce the coefficients of the SDE \eqref{eq: mainSDE}.
Let \(\mathbb{D}\) be the Skorokhod space, i.e. the space of \cadlag functions \(\mathbb{R}_+ \to \mathbb{R}\). We denote the coordinate process by \(\X\), i.e. \(\X(\omega) =\omega\) for all \(\omega \in \D\). Further, on \(\D\) we define \(\sigma\)-field \(\mathcal{D} \triangleq \sigma (\X_t, t \in \mathbb{R}_+)\) and the filtration \(\Df \triangleq (\mathcal{D}_t)_{t \geq 0}\), where \(\mathcal{D}_t \triangleq \bigcap_{s > t} \sigma (\X_r, r \in [0, s])\). For a probability measure \(P\) on \((\D, \mathcal{D})\), \(\Df^P\) denotes the \(P\)-completion of the filtration \(\Df\).

Throughout the paper, we fix two real-valued \(\Df\)-predictable processes \(\w\) and \(\u\) on \(\D\), an initial value \(x_0 \in \mathbb{R}\) and local characteristics \((b, c, F; A)\).

We are in the position to define existence and uniqueness concepts for the SDE \eqref{eq: mainSDE}. Hereby, we adapt terminology from \cite{doi:10.1137/S0040585X97979093,J80}. 
\begin{definition}\begin{enumerate}
		\item[\textup{(i)}] We call \((\B, L)\) a \emph{driving system}, if \(\B = (\Omega, \mathcal{F}, (\mathcal{F}_t)_{t \geq 0}, P)\) is a filtered probability space with right-continuous and complete filtration which supports an SII \(L\) with local characteristics \((b, c, F; A)\). 
		\item[\textup{(ii)}] Let \((\B, L)\) be a driving system. We call a \cadlag adapted process \(X\) on \(\B\) a \emph{solution process} to the SDE \eqref{eq: mainSDE}, if 
		\[
		X = x_0 + \int_0^\cdot \w_s (X) ds + \u (X) \bcdot L, 
		\]
		where it is implicit that the integrals are well-defined, i.e. a.s. \(Z_t \circ X < \infty\) for all~\(t \in \mathbb{R}_+\), where
		\begin{equation}\label{eq: int well-defined} \begin{split}
		Z \triangleq \int_0^\cdot |\w_s| ds &+ \int_0^\cdot \Big( \u_s^2 c_s + \int_{\{|x| \leq 1\}}(1 \wedge |\u_s x|^2) F_s (dx) \Big) d A_s 
		\\&+ \int_0^\cdot \Big| \u_s b_s + \int \big( \u_s x\hspace{0.04cm} \1_{\{|x| \leq 1, |\u_s x| \leq 1\}} - \u_s h(x) \big) F_s (dx) \Big| d A_s,
		\end{split}
		\end{equation}
		see \cite[Theorem III.6.30]{JS}.
		The solution process \(X\) is called a \emph{strong solution process}, if it is adapted to the completed natural filtration of \(L\).
		\item[\textup{(iii)}]
		We say that \emph{weak (strong) existence} holds for the SDE \eqref{eq: mainSDE}, if there exists a driving system which supports a (strong) solution process.
		\item[\textup{(iv)}]
		We say that \emph{strong uniqueness} holds for the SDE \eqref{eq: mainSDE}, if on any driving system there exists up to indistinguishability at most one solution process.
		\item[\textup{(v)}]
		A probability measure \(Q\) on \((\D, \mathcal{D})\) (resp. on \((\D \times \D, \mathcal{D} \otimes \mathcal{D})\)) is called a \emph{(joint) solution measure} for the SDE \eqref{eq: mainSDE}, if there exists a driving system \((\B, L)\) which supports a solution process \(X\) such that \(Q\) is the law of \(X\) (resp. the law of \((X, L)\)).
		\item[\textup{(vi)}]
		We say that \emph{weak (joint) uniqueness} holds for the SDE \eqref{eq: mainSDE}, if there exists at most one (joint) solution measure.
	\end{enumerate}
\end{definition}

Since laws of SIIs are completely characterized by their local characteristics, see \cite[Theorem II.4.25]{JS}, the following theorem follows from \cite[Corollary 8.4]{J80}.
\begin{theorem}\label{prop: jacod}
	For the SDE \eqref{eq: mainSDE} the following are equivalent:
	\begin{enumerate}
		\item[\textup{(i)}] Strong uniqueness and weak existence holds.
		\item[\textup{(ii)}] Weak joint uniqueness and strong existence holds. 
	\end{enumerate}
\end{theorem}
In other words, the equivalence of (i) and (iii) in \Cref{theo: main1intro} holds. 
In the next section we will prove the following:
\begin{theorem}\label{prop: main1}
	For the SDE \eqref{eq: mainSDE} weak uniqueness implies weak joint uniqueness.
\end{theorem}
Together with \Cref{prop: jacod}, this result implies \Cref{theo: main1intro}.

\begin{remark}
	\cite[Theorem 7.2]{J80} shows that the existence of a strong solution on \emph{some} driving system even implies the existence of a strong solution process on \emph{every} driving system. 
\end{remark}
We end this section with comments on possible applications. 
It is interesting to prove analytic conditions implying strong uniqueness. For Brownian SDEs many such conditions are known, see e.g. \cite{MANA:MANA19911510111}. There are less results for SDEs with jumps, see e.g. \cite{bass2004,li2012} for some recent results.
Due to \Cref{theo: main1intro}, instead of verifying strong uniqueness directly, one can deduced it from strong existence and weak uniqueness, where latter is in general better understood than strong uniqueness, see e.g. \cite{kuhn2018} for some recent results for L\'evy driven SDEs. 

As a question for future research, it would be interesting whether one can prove analytic conditions for strong existence without using the Yamada--Watanabe theorem, i.e. strong uniqueness. As pointed out above, this could lead to new results for strong uniqueness. 

We stress that, in general, boundedness, ellipticity and continuity\footnote{even local H\"older continuity with too small exponent} conditions do not suffice for strong existence. For the Brownian case a counterexample is given in \cite{https://doi.org/10.1112/jlms/s2-26.2.335} and for SDEs driven by certain stable L\'evy processes a counterexample is given in \cite{BASS20041}, see also \cite{SPS_2002__36__302_0} and the comment on p. 10 in \cite{bass2004}. 
For the Brownian case let us also mention an example of an SDE with degenerate diffusion coefficient and without strong solution. 
For \(\mu > 0\) consider the system
\begin{align}\label{eq: sticky system}
d X_t = \1_{\{X_t > 0\}} d W_t + \tfrac{1}{2} L^0_t (X), \quad \1_{\{X_t = 0\}} dt = \tfrac{1}{2 \mu} d L^0_t (X), \quad X_0 = x_0 \in \mathbb{R}_+,
\end{align}
where \(W\) is a Brownian motion and \(L^0(X)\) denotes the semimartingale (right) local time of \(X\) in zero. A solution process \(X\) to this system is called \emph{Brownian motion with sticky reflecting in the origin}. 
By \cite[Theorems 5, 6]{engelpeskir14}, the system \eqref{eq: sticky system} is equivalent to the single equation
\begin{align}\label{eq: SDE stick BM}
d X_t = \mu \1_{\{X_t = 0\}} dt + \1_{\{X_t > 0\}} d W_t, \quad X_0 = x_0 \in \mathbb{R}_+, 
\end{align}
which satisfies weak existence and joint weak uniqueness, but not strong existence. 
Consequently, by the Yamada--Watanabe theorem, the SDE \eqref{eq: SDE stick BM} does not satisfy strong uniqueness.
In \cite[Section 3.2]{engelpeskir14} it has been noted that strong uniqueness could be disproved in a direct manner by constructing two distinguishable solution processes, which, together with the dual Yamada--Watanabe theorem, outlines an alternative proof for \cite[Theorem 6]{engelpeskir14}.
It has also been noted that Cherny's \cite{doi:10.1137/S0040585X97979093} version of \Cref{prop: main1} would shorten the proof of \cite[Theorem 5]{engelpeskir14}. 

\section{Proof of \Cref{prop: main1}}\label{sec: 3}
Assume that the SDE \eqref{eq: mainSDE} satisfies weak uniqueness and let \(X\) be a solution process on a driving system \(((\Omega^*, \mathcal{F}^*, (\mathcal{F}^*_t)_{t \geq 0}, P^*), L)\).
Furthermore, let \(((\Omega^o, \mathcal{F}^o, (\mathcal{F}^o_t)_{t \geq 0}, P^o), U)\) be a second driving system, set
\[
\Omega \triangleq \Omega^* \times \Omega^o, \qquad \mathcal{F} \triangleq \mathcal{F}^* \otimes \mathcal{F}^o, \qquad P \triangleq P^* \otimes P^o,
\]
and define \(\mathcal{F}_t\) to be the \(P\)-completion of the \(\sigma\)-field \(\bigcap_{s > t} (\mathcal{F}^*_s \otimes \mathcal{F}^o_s)\). 
In the following \(\mathbb{B} = (\Omega, \mathcal{F}, (\mathcal{F}_t)_{t \geq 0}, P)\) will be our underlying filtered probability space. We extend \(X, L\) and \(U\) to \(\B\) by setting
\[
X (\omega^*, \omega^o) \equiv X(\omega^*), \qquad L (\omega^*, \omega^o) \equiv L(\omega^*), \qquad U(\omega^*, \omega^o) \equiv U(\omega^o)
\]
for \((\omega^*, \omega^o) \in \Omega\).
Due to the results in \cite[Section 10.2 b)]{J79}, \((\B, L)\) and \((\B, U)\) are driving systems and \(X\) is a solution process on \((\B, L)\).
Next, we define a semimartingale \(V\) by
\[
d V_t \triangleq \1_{\{\u_t(X) \not = 0\}} d U_t + \1_{\{\u_t (X) = 0\}} d L_t.
\]
The following lemma is proven after the proof of \Cref{prop: main1} is complete.
\begin{lemma}\label{lem: main1}
	The process \(V\) is an SII with the same local characteristics as \(L\). Moreover, \(V\) is independent of \(X\).
\end{lemma}
Since
\[
d L_t = \frac{\1_{\{\u_t (X) \not = 0\}} (dX_t - \w_t (X) dt)}{\u_t (X)} + \1_{\{\u_t (X) = 0\}} d V_t,
\]
the distribution of \((X, L)\) is completely determined by the distribution of \((X, V)\) and \Cref{lem: main1} yields that weak joint uniqueness is implied by weak uniqueness. The proof of \Cref{prop: main1} is complete.
\qed
\begin{remark}
	In case \(\u\) is non-degenerate, i.e. \(\u \not = 0\), it is clear that \(V = U\) and \Cref{lem: main1} becomes trivial. In particular, for the key argument it is not necessary to introduce the auxilliary process \(U\), see \eqref{eq: intro} in the introduction.
\end{remark}

It remains to prove \Cref{lem: main1}:
\\

\noindent
\emph{Proof of \Cref{lem: main1}:}
\emph{Step 1.} 
Recall that the local characteristics of \(L\) are denoted by \((b, c, F; A)\).
Our first step is to show that \(V\) is an SII with local characteristics \((b, c, F; A)\) by computing its semimartingale characteristics \((B^V, C^V, \nu^V)\).
We start with the first characteristic \(B^V\). 
For \(\h(x) \triangleq x - h(x)\), we have
\begin{align*}
\sum_{s \leq \cdot} \h (\Delta V_s) &= \sum_{s \leq \cdot} \h(\1_{\{\u_s(X) \not= 0\}} \Delta U_s + \1_{\{\u_s(X) = 0\}} \Delta L_s) (\1_{\{\u_s (X) \not= 0\}} + \1_{\{\u_t (X) = 0\}})
\\&= \sum_{s \leq \cdot} \h( \Delta U_s) \1_{\{\u_s(X) \not= 0\}} + \sum_{s \leq \cdot} \h(\Delta L_s) \1_{\{\u_s (X) = 0\}}.
\end{align*}
The definition of the first semimartingale characteristic shows that the processes 
\begin{align*}
U - \sum_{s \leq \cdot} \h(\Delta U_s) - \int_0^\cdot b_s dA_s, \qquad 
L - \sum_{s \leq \cdot} \h(\Delta L_s) - \int_0^\cdot b_s dA_s 
\end{align*}
are local martingales.
Hence, the process
\begin{align*}
V - \sum_{s \leq \cdot} \h(\Delta V_s) - \int_0^\cdot b_s d A_s &= \1_{\{\u (X) \not = 0\}} \bcdot \Big( U - \sum_{s \leq \cdot} \h(\Delta U_s) - \int_0^\cdot b_s d A_s \Big) \\&\qquad\quad+ \1_{\{\u (X) = 0\}} \bcdot \Big( L - \sum_{s \leq \cdot} \h(\Delta L_s) - \int_0^\cdot b_s dA_s\Big) 
\end{align*}
is a local martingale, too.
This implies that \(d B^V_t = b_t dA_t\).
Next, we compute the second characteristic \(C^V\).
Using that
\(
d V^c_t = \1_{\{\u_t(X) \not = 0\}} d U^c_t + \1_{\{\u_t (X) = 0\}} d L^c_t,
\)
we obtain 
\[
d C^V_t = d [V^c, V^c]_t = \1_{\{\u_t(X) \not= 0\}} d [U^c, U^c]_t +  \1_{\{\u_t(X) = 0\}} d [L^c, L^c]_t 
= c_t dA_t.
\]
Finally, we compute the third characteristic \(\nu^V\). 
For every bounded stopping time \(T\) and any bounded Borel function \(G \colon \mathbb{R} \to \mathbb{R}\) which vanishes in a neighborhood of the origin, we obtain
\begin{align*}
E \Big[ \sum_{s \leq T} G(\Delta V_s)\Big] &= E \Big[ \sum_{s \leq T} \big(G(\Delta U_s) \1_{\{\u_s(X) \not= 0\}} + G (\Delta L_s)\1_{\{\u_s(X) = 0\}}\big) \Big]
\\&= E \Big[ \int_0^T \hspace{-0.125cm}\int \big(G(x) \1_{\{\u_s (X) \not = 0\}} +G(x) \1_{\{\u_s (X) = 0\}} \big) F_s (dx) d A_s \Big]
\\&= E \Big[ \int_0^T \hspace{-0.125cm}\int G(x) F_s (dx) d A_s\Big].
\end{align*}
This shows that
\(
\nu^V(dx, dt) = F_t(dx) dA_t.
\)
In summary, the process \(V\) is an SII with local characteristics \((b, c, F; A)\).
\\

\noindent
\emph{Step 2.} In this step we show that \(V\) and \(X\) are independent. Hereby, we borrow ideas used in the proof of \cite[Theorem 4.10.1]{EK}.
We start with three preparatory lemmata.
\begin{lemma}\label{lem: biv (U, L)}
	The bivariate process \((U, L)\) is a two-dimensional SII with local characteristics \((b^{(U, L)}, c^{(U, L)}, F^{(U, L)}; A)\) given by
	\[
	b^{(U, L)} = (b, b), \quad c^{(U, L)} = \begin{pmatrix} c&0\\0&c\end{pmatrix}, \quad 
	F^{(U, L)} (dx, dy) = F(dx) \delta_0 (dy) + \delta_0 (dx) F(dy).
	\]
\end{lemma}
\begin{proof}
	It is clear that \((U, L)\) is a two-dimensional semimartingale. 
	Fix \(0 \leq s < t\). By the construction of \(\B\) and a monotone class argument, we see that \((U_t - U_s, L_t - L_s)\) is independent of \(\mathcal{F}^*_s \otimes \mathcal{F}^o_s\). Let \(f \colon \mathbb{R}^2 \to \mathbb{R}_+\) be bounded and continuous and take \(G \in \mathcal{F}_s\). For every \(\varepsilon \in (0, t - s)\) we have
	\[
	E \big[ f(U_t - U_{s + \varepsilon}, L_t - L_{s + \varepsilon}) \1_G \big] = E\big[ f(U_t - U_{s + \varepsilon}, L_t - L_{s + \varepsilon})\big] P(G).
	\]
	Thus, using that \(U\) and \(L\) have right-continuous paths, that for every closed set \(F \subseteq \mathbb{R}^2\) there exists a uniformly bounded sequence \((f_n)_{n \in \mathbb{N}}\) of non-negative continuous functions such that \(f_n (x) \to \1_F (x)\), and a monotone class argument, we conclude that \((U, L)\) has independent increments relative to the filtration \((\mathcal{F}_t)_{t \geq 0}\).
	Finally, the structure of the characteristics follows from the independence of \(U\) and \(L\), the L\'evy--Khinchine formula, see \cite[Theorem II.4.15]{JS}, and the uniqueness lemma \cite[Lemma~II.2.44]{JS}.
\end{proof}

For \(f \in C^2_b(\mathbb{R})\), which is the space of bounded twice continuously differentiable functions with bounded first and second derivative, and \((x, t) \in \mathbb{R} \times \mathbb{R}_+\) we set
\[
\mathcal{L}f (x, t) \triangleq b_t f'(x) + \frac{c_t}{2} f''(x) + \int \big( f(x + y) - f(x) - h(y) f'(x) \big) F_t(dy).
\]
The next lemma is a martingale characterization for SIIs.
\begin{lemma}\label{lem: main3}
	On a filtered probability space with right-continuous filtration, let \(Y\) be an adapted process with \cadlag paths and starting value \(Y_0 = 0\).
	Then, the following are equivalent: 
	\begin{enumerate} \item[\textup{(i)}] \(Y\) is an SII with local characteristics \((b, c, F; A)\).
		\item[\textup{(ii)}] For all \(f \in C^2_b(\mathbb{R})\) with \(\inf_{x \in \mathbb{R}} f(x) > 0\) the process
		\begin{align}\label{eq: Mf}
		M^f \triangleq \frac{f(Y)}{f (0)} \exp \Big( - \int_0^\cdot \frac{\mathcal{L}f (Y_s, s) dA_s}{f (Y_s)} \Big)
		\end{align}
		is a martingale.
	\end{enumerate}
\end{lemma}
\begin{proof}
	Suppose that (i) holds and let \(f \in C^2_b(\mathbb{R})\) with \(\inf_{x \in \mathbb{R}} f (x) > 0\). Integration by parts yields that 
	\begin{align}\label{eq: ito Mf}
	d M^f_t = \frac{1}{f(0)}  \exp \Big( - \int_0^t \frac{\mathcal{L}f (Y_s, s) d A_s}{f (Y_s)}\Big) (d f(Y_t) - \mathcal{L} f(Y_t, t) dA_t).
	\end{align}
	By \cite[Theorem II.2.42]{JS}, the process \(f (Y) - \int_0^\cdot \mathcal{L} f(Y_s, s) dA_s\) is a local martingale. Thus, also \(M^f\) is a local martingale. As \(h\) is a truncation function, there exists a constant \(\varepsilon > 0\) such that \(h(x) = x\) whenever \(x \in (- \varepsilon,\varepsilon)\). Furthermore, by Taylor's theorem, 
	\[
	| f(x + y) - f(x) - f'(x) y | \leq \tfrac{1}{2} \|f''\|_\infty |y|^2, \quad y \in \mathbb{R}.
	\]
	Thus, \(\int_0^\cdot \int (1 \wedge |y|^2) F_s (dy) dA_s < \infty\) implies that the process \(M^f\) is bounded on bounded time intervals and consequently, a martingale. We conclude that (ii) holds.
	
	Conversely, assume that (ii) holds. Again, let \(f \in C^2_b(\mathbb{R})\) with \(\inf_{x \in \mathbb{R}} f (x) > 0\).  Then, integration by parts yields that 
	\[
	d f (Y_t)  - \mathcal{L} f (Y_t, t) dA_t = f (0) \exp \Big(\int_0^t \frac{\mathcal{L}f (Y_s, s) d A_s}{f (Y_s)}\Big) d M^f_t.
	\]
	Hence, \(f (Y) - \int_0^\cdot \mathcal{L} f (Y_s, s) dA_s\) is a local martingale. Applying this observation with \(f(x) = 2 + \sin (u x)\) and \(f (x) = 2 + \cos (u x)\) for \(u \in \mathbb{R}\) yields that b) of \cite[Theorem~II.2.42]{JS} holds and (i) follows. 
\end{proof}
Next, we also give a martingale characterization for the set of solution measures to the SDE \eqref{eq: mainSDE}. 
For \(g \in C^2_b(\mathbb{R})\) and \((\omega, s) \in \D \times \mathbb{R}_+\) we set 
\begin{align*}
\mathcal{K} g (\omega, s) \triangleq &\ \u_s (\omega) b_s g' (\omega (s-)) + \frac{\u^2_s (\omega) c_s}{2} g'' (\omega (s-)) \\&\ \ + \int ( g(\omega(s-) + \u_s (\omega) x) - g(\omega(s-)) - \u_s (\omega) h(x) g' (\omega (s-)) ) F_s(dx).
\end{align*}
Similar to the proof of \cite[Theorem 14.80]{J79}, the following lemma follows from representation results for continuous local martingales and integer valued random measures. We refer to \cite{kurtz10} for an approach not relying on representation theorems.

\begin{lemma}\label{lem: main2}
	A probability measure \(Q\) on \((\D, \mathcal{D})\) is a solution measure to the SDE \eqref{eq: mainSDE} if and only if \(Q(\X_0 = x_0) = 1\), \(Z\) as defined in \eqref{eq: int well-defined} is \(Q\)-a.s. finite, and for all \(g \in C^2_b (\mathbb{R})\) the process
	\begin{align}\label{eq: Kf}
	K^g \triangleq g(\X) - g(x_0) - \int_0^\cdot \w_s (\X)g' (\X_{s}) ds -  \int_0^\cdot \mathcal{K} g(\X, s) d A_s
	\end{align}
	is a local \((\Df^Q, Q)\)-martingale. Furthermore, for every solution process \(X\) to the SDE~\eqref{eq: mainSDE} the process \(K^g \circ X\) is a local martingale on the corresponding driving system.
\end{lemma}
\begin{proof}
	Let \(X\) be a solution process to the SDE \eqref{eq: mainSDE} and denote its law by \(Q\). Then, \cite[Theorem II.2.42]{JS} and \cite[Lemma 3]{kallsenshiryaev02} yield that the process \(K^g \circ X\) is a local martingale for each \(g \in C^2_b (\mathbb{R})\). Due to \cite[Remark 10.40]{J79} this local martingale property transfers to the canonical space, i.e. \(K^g\) is a local \((\Df^Q, Q)\)-martingale. Hence, the final claim and the \emph{only if} implication are proven.
	
	Conversely, assume that \(Q\) is as in the \emph{if} implication. Recalling again \cite[Lemma 3]{kallsenshiryaev02}, we deduce from \cite[Theorems~II.2.34,~II.2.42]{JS} and \cite[Theorems 1, 2]{Kabanov_1981} that, possibly on a standard extension of the canonical space endowed with \(Q\), there exists a continuous local martingale \(L^c\) with quadratic variation \([L^c, L^c] = \int_0^\cdot c_s d A_s\) and an integer valued random measure \(\mu (dx, dt)\) with intensity measure \(\nu (dx, dt) = F_t (dx) d A_t\) such that 
	\begin{align*}
	d \X_t = \w_t (\X) dt &+  \u_t (\X) \Big( b_t + \int \big(x \1_{\{|\u_t (\X) x| \leq 1\}} - h(x) \big) F_t (dx) \Big) d A_t
	\\&+ \u_t (\X) d L^c_t + \int \u_t (\X) x \1_{\{|\u_t (\X) x| \leq 1\}} (\mu - \nu)(dx, dt)
	\\&+ \int \u_t (X) x \1_{\{|\u_t(\X) x| > 1\}} \mu (dx, dt).
	\end{align*}
	Rearranging yields that 
	\(
	d \X_t = \w_t (\X) dt + \u_t (\X) d L_t, 
	\)
	where
	\begin{equation*} 
	\begin{split}
	d L_t \triangleq b_t d A_t + d L^c_t + \int h(x) (\mu - \nu) (dx, dt) + \int (x - h(x)) \mu (dx, dt).
	\end{split}
	\end{equation*}
	It is easy to see that \(L\) is an SII with local characteristics \((b, c, F; A)\) and consequently, the \emph{if} implication holds.
	The proof is complete.
\end{proof}

We are in the position to prove the independence of \(V\) and \(X\).
Take \(f, g \in C^2_b(\mathbb{R})\) with \(\inf_{x \in \mathbb{R}} f(x) > 0\) and define \(M^f\) and \(K^g\) as in \eqref{eq: Mf} and \eqref{eq: Kf} with \(Y\) replaced by \(V\) and \(\X\) replaced by \(X\). As shown in Step 1, \(V\) is an SII with local characteristics \((b, c, F; A)\). Hence, \(M^f\) is a martingale by \Cref{lem: main3}. Similarly, because \(X\) is a solution process to the SDE \eqref{eq: mainSDE}, \(K^g\) is a local martingale by \Cref{lem: main2}.
We now show that \([M^f, K^g] = 0.\)
Recalling \eqref{eq: ito Mf} and the proof of a) \(\Rightarrow\) c) in \cite[Theorem II.2.42]{JS}, we see that 
\[
(M^f)^c = \frac{f'(V)}{f(0)} \exp \Big( - \int_0^\cdot \frac{\mathcal{L} f(V_s, s) dA_s}{f(V_s)} \Big)  \bcdot V^c.
\]
Similarly, \((K^g)^c = g'(X) \u (X) \bcdot L^c\). 
Since \([U^c, L^c] = 0\) by \eqref{lem: biv (U, L)}, we obtain
\[
d [V^c, \u (X) \bcdot L^c]_t = \u_t(X) \1_{\{\u_t (X) \not = 0\}} d [U^c, L^c]_t + \u_t (X) \1_{\{\u_t(X) = 0\}} d [L^c, L^c]_t = 0,
\]
which implies that \([(M^f)^c, (K^g)^c] = 0\).
Using the formula for the third characteristic of \((U, L)\) as given in \eqref{lem: biv (U, L)} yields that 
\begin{align*} 
E \Big[ \sum_{s \leq \cdot} | \Delta U_s \Delta L_s | \Big] =  \int_0^\cdot \hspace{-0.1cm} \int |xy| F^{(U, L)}_s(dx, dy) dA_s = 0.
\end{align*}
In other words, \(U\) and \(L\) cannot jump at the same time, which implies that
\[\u(X) \not = 0,\Delta L \not = 0 \ \Longrightarrow \ \Delta V = \1_{\{\u (X) \not= 0\}} \Delta U + \1_{\{\u (X) = 0\}} \Delta L = 0.\] Hence, we obtain
\[
\u(X) \not = 0,\Delta L \not = 0 \ \Longrightarrow \ \Delta M^f = \frac{\Delta f (V)}{f(0)} \exp \Big( - \int_0^\cdot \frac{\mathcal{L} f(V_s, s) dA_s}{f(V_s)}\Big) = 0.
\]
Furthermore, because 
\(
\Delta X = \u (X) \Delta L, 
\)
we have 
\[
\u (X) = 0 \text{ or } \Delta L = 0 \ \Longrightarrow \ \Delta K^g = \Delta g(X) = 0.
\]
Putting these pieces together, we conclude that
\(
\sum_{s \leq \cdot} \Delta M^f_s \Delta K^g_s = 0
\)
and therefore \[[M^f, K^g] = [(M^f)^c, (K^g)^c] +\sum_{s \leq \cdot } \Delta M^f_s \Delta K^g_s = 0.\] 

For \(n \in \mathbb{N}\) we define 
\[
T_n \triangleq \inf (t \in \mathbb{R}_+ \colon |K^g_t| > n), \qquad K^{g, n} \triangleq K^g_{\cdot \wedge T_n}.
\]
As \(K^g\) has bounded jumps, \(K^{g, n}\) is bounded on bounded time intervals. Thus, by integration by parts, \([M^{f}, K^{g, n}] = [M^f, K^g]_{\cdot \wedge T_n} = 0\) yields that the process \(M^f K^{g, n}\) is a martingale which is bounded on bounded time intervals.

Now, fix a bounded stopping time \(S\) and 
define a measure \(Q\) as follows:
\[
Q (G) \triangleq E^P \big[ M^f_S \1_G \big], \quad G \in \mathcal{F}.
\]
Clearly, as \(M^f\) is a \(P\)-martingale with \(M^f_0 = 1\), \(Q\) is a probability measure by the optional stopping theorem. Moreover, because \(Q \sim P\), we have \(Q(X_0 = x_0) = 1\) and \(Q\)-a.s. \(Z_t \circ X < \infty\) for all \(t \in \mathbb{R}_+\), where \(Z\) is defined in \eqref{eq: int well-defined}.
As \(M^f, K^{g, n}\) and \(M^f K^{g, n}\) are \(P\)-martingales, we obtain for every bounded stopping time \(T\) that 
\begin{align*}
E^Q \big[ K^{g, n}_T \big] &= E^P \big[M^f_S K^{g, n}_T\big]
\\&=E^P \big[ M^f_S \1_{\{S \leq T\}} E^P\big[K^{g, n}_T | \mathcal{F}_{S \wedge T} \big] + K^{g, n}_T \1_{\{T < S\}} E^P\big[M^f_S | \mathcal{F}_{S \wedge T} \big]\big]
\\&=E^P \big[ M^f_SK^{g, n}_{S \wedge T}\1_{\{S \leq T\}} + K^{g, n}_T M^f_{S \wedge T}\1_{\{T < S\}}\big]
\\&=E^P \big[ M^f_{S \wedge T} K^{g, n}_{S \wedge T} \big] =
0. 
\end{align*}
Thus, because the stopping time \(T\) was arbitrary, \(K^{g, n}\) is a \(Q\)-martingale. Furthermore, since \(g\) was arbitrary and \(T_n \nearrow \infty\) as \(n \to \infty\), \Cref{lem: main2} and \cite[Remark 10.40]{J79} show that the push-forward \(Q \circ X^{-1}\) is a solution measure to the SDE \eqref{eq: mainSDE}. Consequently, by the weak uniqueness assumption, \(P \circ X^{-1} = Q \circ X^{-1}\). 

Next, take \(F \in \sigma (X_t, t \in \mathbb{R}_+)\) such that \(P (F) > 0\) and set 
\[
Q^* (G) \triangleq \frac{P(G \cap F)}{P(F)},\quad G \in \mathcal{F}.
\]
Using that \(P(F) = Q(F)\), we obtain that
\[
E^{Q^*} \big[ M^f_S \big] = \frac{Q(F) }{P(F)} =  1.
\]
Thus, as the stopping time \(S\) was arbitrary, \(M^f\) is a \(Q^*\)-martingale. 
Since \(f\) was arbitrary, we deduce from \Cref{lem: main3} that \(Q^* \circ V^{-1} = P \circ V^{-1}\) and consequently, for every \(G \in \sigma (V_t, t \in \mathbb{R}_+)\)
\begin{align*}
P(G \cap F) = Q^*(G) P(F) = P(G) P(F).
\end{align*}
As this equality holds trivially for all \(F \in \sigma (X_t, t \in \mathbb{R}_+)\) with \(P(F) = 0\), we conclude that \(V\) and \(X\) are independent. The proof is complete.
\qed

\bibliographystyle{plain}

\end{document}